\documentclass[11pt]{amsart}

\usepackage{hyperref}
\hypersetup{nesting=true,debug=true,naturalnames=true}
\usepackage{graphicx,amssymb,upref}
\usepackage[usenames,dvipsnames]{pstricks}
\usepackage{epsfig}
\usepackage{color}
\usepackage[latin1]{inputenc}
\usepackage[T1]{fontenc}
\usepackage{amsmath, amsthm}
\usepackage[english]{babel}
\usepackage{amssymb}
\usepackage{latexsym}
\usepackage{graphicx,amssymb,upref}
\usepackage[all]{xy}
\usepackage{tikz}
\usetikzlibrary{matrix}

\date{\today}

\newtheorem{theorem}{Theorem}[section]
\newtheorem{lemma}[theorem]{Lemma}

\newtheorem{corollary}[theorem]{Corollary}
\theoremstyle{definition}
\newtheorem{definition}[theorem]{Definition}
\newtheorem{example}[theorem]{Example}

\newtheorem{remark}[theorem]{Remark}

\newcommand{\ot}{\otimes}
\newcommand{\co}{\circ}

\title[Modules over invertible 1-cocycles]{Modules over invertible 1-cocycles}

\title{Modules over invertible 1-cocycles} 

\begin{document}
	
	\maketitle
	
	\begin{center}
	{\bf Jos\'e Manuel Fern\'andez Vilaboa$^{1}$, Ram\'on
	Gonz\'{a}lez Rodr\'{\i}guez$^{2}$, Brais Ramos P\'erez$^{3}$ and Ana Bel\'en Rodr\'{\i}guez Raposo$^{4}$}.
	\end{center}
	
	\vspace{0.4cm}
	\begin{center}
	{\small $^{1}$ [https://orcid.org/0000-0002-5995-7961].}
	\end{center}
	\begin{center}
	{\small  CITMAga, 15782 Santiago de Compostela, Spain.}
	\end{center}
	\begin{center}
	{\small  Universidade de Santiago de Compostela. Departamento de Matem\'aticas,  Facultade de Matem\'aticas, E-15771 Santiago de Compostela, Spain. 
	\\ email: josemanuel.fernandez@usc.es.}
	\end{center}
	\vspace{0.2cm}
	
	\begin{center}
	{\small $^{2}$ [https://orcid.org/0000-0003-3061-6685].}
	\end{center}
	\begin{center}
	{\small  CITMAga, 15782 Santiago de Compostela, Spain.}
	\end{center}
	\begin{center}
	{\small  Universidade de Vigo, Departamento de Matem\'{a}tica Aplicada II,  E. E. Telecomunicaci\'on,
	E-36310  Vigo, Spain.
	\\email: rgon@dma.uvigo.es}
	\end{center}


   \begin{center}
   	{\small $^{3}$ [https://orcid.org/0009-0006-3912-4483].}
   \end{center}
	\begin{center}
	{\small  CITMAga, 15782 Santiago de Compostela, Spain. \\}
	\end{center}
    \begin{center}
	{\small  Universidade de Santiago de Compostela. Departamento de Matem\'aticas,  Facultade de Matem\'aticas, E-15771 Santiago de Compostela, Spain. 
	\\email: braisramos.perez@usc.es}
	\end{center}
	\vspace{0.2cm}
	
	\begin{center}
	{\small $^{4}$ [https://orcid.org/0000-0002-8719-5159]}
	\end{center}
	\begin{center}
	{\small Universidade de Santiago de Compostela, Departamento de Did\'acticas Aplicadas, Facultade C. C. Educaci\'on, E-15782 Santiago de Compostela, Spain.
	\\email: anabelen.rodriguez.raposo@usc.es}
	\end{center}

\begin{abstract}  
 In this paper we introduce in a braided setting the notion of left module for an invertible 1-cocycle and we prove some categorical equivalences between categories of modules associated to an invertible 1-cocycle and categories of modules, in the sense of \cite{RGON}, associated to Hopf braces. 
 \end{abstract} 

\vspace{0.2cm}

{\sc Keywords}: Braided monoidal category, Hopf algebra,  Hopf brace,  invertible 1-cocycles,  module. 

{\sc MSC2020}: 18M05, 16T05, 18G45, 16S40.

\section*{Introduction}

Hopf braces were introduced recently  in \cite{AGV} as the linearisation of skew braces given in \cite{GV}. A skew brace consists of two different group structures $(T, \diamond)$ and $(T, \circ)$ on the same set $T$ and that satisfy $\forall\; a, b, c\in T$ the compatibility condition
$$a\circ (b\diamond c)=(a\circ b)\diamond a^{\diamond}\diamond (a\circ c),$$
where $a^{\diamond}$ denotes the inverse with respect to $\diamond$. Thus, a Hopf brace consists on two structures  of Hopf algebras defined on the same object that share a common coalgebra structure and satisfying the same compatibility condition that generalizes the previous identity.  The relevance of these structures comes from the fact that provide solutions of the Yang-Baxter equation and, as was pointed in \cite{AGV}, they give the right setting for considering left symmetric algebras as Lie-theoretical analogs of the notion of brace introduced by W. Rump in \cite{Rump}. On the other hand, it is also interesting to note that there exists a deep link between Hopf braces and  invertible 1-cocycles. In fact, the category of Hopf braces  is equivalent to the category of invertible 1-cocycles. Thus, invertible 1-cocycles are nothing more than coalgebra isomorphisms between Hopf algebras that share the underlying coalgebra and related by a module algebra structure. 

As long as we are dealing with Hopf type structures, we are somehow forced to have a deep look of their categories of modules for a complete overview. For example, in \cite{RGON} the author introduces the category of left modules for a Hopf brace  in order to prove Fundamental Theorem of Hopf modules (see, for example \cite{Abe} and  \cite{SW}) in the Hopf brace setting. This notion of module for a Hopf brace  is weaker than the one introduced by H. Zhu in \cite{Zhu} and in the cocommutative setting both notions are equivalent. The main difference between the two definitions is the following: The Hopf brace with the two associated products is an example of module in the sense of \cite{RGON} while with the definition proposed by Zhu that property only holds when the underlying object of the Hopf brace endowed with a particular action  is an object belonging to a  class of cocommutativity in the sense of \cite{CCH}. In other words, under certain circumstances, for example, the lack of cocommutativity, the category of left modules over a Hopf brace introduced by Zhu may not contain the trivial object as it happens always in the case of Hopf algebras.

Therefore  the main objective of this article is to find the appropriate notion of module associated to an invertible 1-cocycle that allows to extend the previous categorical equivalence  between  Hopf braces and invertible 1-cocycles to their categories of  modules, taking into account that for us the notion of module associated with a Hopf brace is the one  introduced in \cite{RGON}. 

The paper is organized as follows: The first section contains the basic notions that we will need in the rest of the paper and the main results  are in the second section. More concretely, working in a braided setting, in Section 2 we introduce the notion of module for an invertible 1-cocycle and the category of these objects (see Definition \ref{def-pi-module}). After doing this we prove some functorial results and we show that under symmetry and cocommutativity conditions the category of modules associated to an invertible 1-cocycle is symmetric monoidal (see Theorem \ref{th-mon-cat}). Finally, in Theorems 
\ref{th-mon-cat21},  \ref{pHpi} and Corollary \ref{crpHpi} we obtain  the desired categorical equivalences between categories of modules associated to an invertible 1-cocycle and categories of modules associated to a Hopf brace. 

\section{Preliminaries}

Throughout this paper ${\sf  C}$ denotes a strict braided monoidal category with tensor product $\ot$, unit object $K$ and braiding $c$. Recall that a monoidal category is a category ${\sf  C}$ together with a functor $\ot :{\sf  C}\times {\sf  C}\rightarrow {\sf  C}$, called tensor product, an object $K$ of ${\sf C}$, called the unit object, and  families of natural isomorphisms 
$$a_{M,N,P}:(M\ot N)\ot P\rightarrow M\ot (N\ot P),\;\;\;r_{M}:M\ot K\rightarrow M, \;\;\; l_{M}:K\ot M\rightarrow M,$$
in ${\sf  C}$, called  associativity, right unit and left unit constraints, respectively, satisfying the Pentagon Axiom and the Triangle Axiom, i.e.,
$$a_{M,N, P\ot Q}\co a_{M\ot N,P,Q}= (id_{M}\ot a_{N,P,Q})\co a_{M,N\ot P,Q}\co (a_{M,N,P}\ot id_{Q}),$$
$$(id_{M}\ot l_{N})\co a_{M,K,N}=r_{M}\ot id_{N},$$
where for each object $X$ in ${\sf  C}$, $id_{X}$ denotes the identity morphism of $X$ (see  \cite{Mac}). A monoidal category is called strict if the constraints of the previous paragraph are identities. It is a well-known  fact (see for example \cite{Christian}) that every non-strict monoidal category is monoidal equivalent to a strict one. This lets us to treat monoidal categories as if they were strict and, as a consequence, the results proved in a strict setting hold for every non-strict  monoidal category, for example the category ${\mathbb F}$-{\sf Vect} of vector spaces over a field ${\mathbb F}$, or the category $R$-{\sf Mod} of left modules over a commutative ring $R$.
For
simplicity of notation, given objects $M$, $N$, $P$ in ${\sf  C}$ and a morphism $f:M\rightarrow N$, in most cases we will write $P\ot f$ for $id_{P}\ot f$ and $f \ot P$ for $f\ot id_{P}$.

A braiding for a strict monoidal category ${\sf  C}$ is a natural family of isomorphisms 
$$c_{M,N}:M\ot N\rightarrow N\ot M$$ subject to the conditions 
$$
c_{M,N\ot P}= (N\ot c_{M,P})\co (c_{M,N}\ot P),\;\;
c_{M\ot N, P}= (c_{M,P}\ot N)\co (M\ot c_{N,P}).
$$

A strict braided monoidal category ${\sf  C}$ is a strict monoidal category with a braiding.  Note that, as a consequence of the definition, the equalities $c_{M,K}=c_{K,M}=id_{M}$ hold, for all object  $M$ of ${\sf  C}$.  If the braiding satisfies that  $c_{N,M}\co c_{M,N}=id_{M\ot N},$ for all $M$, $N$ in ${\sf  C}$, we will say that ${\sf C}$  is symmetric. In this case, we call the braiding $c$ a symmetry for the category ${\sf  C}$.

\begin{definition}
{\rm 
An algebra in ${\sf  C}$ is a triple $A=(A, \eta_{A},
\mu_{A})$ where $A$ is an object in ${\sf  C}$ and
 $\eta_{A}:K\rightarrow A$ (unit), $\mu_{A}:A\otimes A
\rightarrow A$ (product) are morphisms in ${\sf  C}$ such that
$\mu_{A}\circ (A\otimes \eta_{A})=id_{A}=\mu_{A}\circ
(\eta_{A}\otimes A)$, $\mu_{A}\circ (A\otimes \mu_{A})=\mu_{A}\circ
(\mu_{A}\otimes A)$. Given two algebras $A= (A, \eta_{A}, \mu_{A})$
and $B=(B, \eta_{B}, \mu_{B})$, a morphism  $f:A\rightarrow B$ in {\sf  C} is an algebra morphism if $\mu_{B}\circ (f\otimes f)=f\circ \mu_{A}$, $ f\circ
\eta_{A}= \eta_{B}$. 

If  $A$, $B$ are algebras in ${\sf  C}$, the tensor product
$A\otimes B$ is also an algebra in
${\sf  C}$ where
$\eta_{A\otimes B}=\eta_{A}\otimes \eta_{B}$ and $\mu_{A\otimes
	B}=(\mu_{A}\otimes \mu_{B})\circ (A\otimes c_{B,A}\otimes B).$
}
\end{definition}

\begin{definition}
{\rm 
A coalgebra  in ${\sf  C}$ is a triple ${D} = (D,
\varepsilon_{D}, \delta_{D})$ where $D$ is an object in ${\sf
C}$ and $\varepsilon_{D}: D\rightarrow K$ (counit),
$\delta_{D}:D\rightarrow D\otimes D$ (coproduct) are morphisms in
${\sf  C}$ such that $(\varepsilon_{D}\otimes D)\circ
\delta_{D}= id_{D}=(D\otimes \varepsilon_{D})\circ \delta_{D}$,
$(\delta_{D}\otimes D)\circ \delta_{D}=
 (D\otimes \delta_{D})\circ \delta_{D}.$ If ${D} = (D, \varepsilon_{D},
 \delta_{D})$ and
${ E} = (E, \varepsilon_{E}, \delta_{E})$ are coalgebras, a morphism 
$f:D\rightarrow E$ in  {\sf  C} is a coalgebra morphism if $(f\otimes f)\circ
\delta_{D} =\delta_{E}\circ f$, $\varepsilon_{E}\circ f
=\varepsilon_{D}.$ 

Given  $D$, $E$ coalgebras in ${\sf  C}$, the tensor product $D\otimes E$ is a
coalgebra in ${\sf  C}$ where $\varepsilon_{D\otimes
E}=\varepsilon_{D}\otimes \varepsilon_{E}$ and $\delta_{D\otimes
E}=(D\otimes c_{D,E}\otimes E)\circ( \delta_{D}\otimes \delta_{E}).$
}
\end{definition}

\begin{definition}
	{\rm 
 Let ${D} = (D, \varepsilon_{D},
\delta_{D})$ be a coalgebra and let $A=(A, \eta_{A}, \mu_{A})$ be an
algebra. By ${\mathcal  H}(D,A)$ we denote the set of morphisms
$f:D\rightarrow A$ in ${\sf  C}$. With the convolution operation
$f\ast g= \mu_{A}\circ (f\otimes g)\circ \delta_{D}$, ${\mathcal  H}(D,A)$ is an algebra where the unit element is $\eta_{A}\circ \varepsilon_{D}=\varepsilon_{D}\otimes \eta_{A}$.
}
\end{definition}

\begin{definition}

 Let  $A$ be an algebra. The pair
$(M,\varphi_{M})$ is a left $A$-module if $M$ is an object in
${\sf  C}$ and $\varphi_{M}:A\otimes M\rightarrow M$ is a morphism
in ${\sf  C}$ satisfying $\varphi_{M}\circ(
\eta_{A}\ot M)=id_{M}$, $\varphi_{M}\circ (A\ot \varphi_{M})=\varphi_{M}\circ
(\mu_{A}\ot M)$. Given two left ${A}$-modules $(M,\varphi_{M})$
and $(N,\varphi_{N})$, $f:M\rightarrow N$ is a morphism of left
${A}$-modules if $\varphi_{N}\circ (A\ot f)=f\circ \varphi_{M}$.  

The  composition of morphisms of left $A$-modules is a morphism of left $A$-modules. Then left $A$-modules form a category that we will denote by $\;_{\sf A}${\sf Mod}.

\end{definition}

\begin{definition}
{\rm 
We say that $H$ is a
bialgebra  in ${\sf  C}$ if $(H, \eta_{H}, \mu_{H})$ is an
algebra, $(H, \varepsilon_{H}, \delta_{H})$ is a coalgebra, and
$\varepsilon_{H}$ and $\delta_{H}$ are algebra morphisms
(equivalently, $\eta_{H}$ and $\mu_{H}$ are coalgebra morphisms). Moreover, if there exists a morphism $\lambda_{H}:H\rightarrow H$ in ${\sf  C}$,
called the antipode of $H$, satisfying that $\lambda_{H}$ is the inverse of $id_{H}$ in ${\mathcal  H}(H,H)$, i.e., 
\begin{equation}
\label{antipode}
id_{H}\ast \lambda_{H}= \eta_{H}\circ \varepsilon_{H}= \lambda_{H}\ast id_{H},
\end{equation}
we say that $H$ is a Hopf algebra.

}
\end{definition}

If $H$ is a Hopf algebra,  the antipode is antimultiplicative and anticomultiplicative 
$$
\lambda_{H}\co \mu_{H}=  \mu_{H}\co (\lambda_{H}\ot \lambda_{H})\co c_{H,H},\;\;\;\; \delta_{H}\co \lambda_{H}=c_{H,H}\co (\lambda_{H}\ot \lambda_{H})\co \delta_{H}, 
$$
and leaves the unit and counit invariant, i.e., 
$$
\lambda_{H}\co \eta_{H}=  \eta_{H},\;\; \varepsilon_{H}\co \lambda_{H}=\varepsilon_{H}.
$$

A morphism of Hopf algebras is an algebra-coalgebra morphism. Note that, if $f:H\rightarrow D$ is a Hopf algebra morphism the following equality holds:
$$
\lambda_{D}\co f=f\co \lambda_{H}.
$$

With the composition of morphisms in {\sf C} we can define a category whose objects are  Hopf algebras  and whose morphisms are morphisms of Hopf algebras. We denote this category by ${\sf  Hopf}$.

A Hopf algebra is commutative if $\mu_{H}\co c_{H,H}=\mu_{H}$ and cocommutative if $c_{H,H}\co \delta_{H}=\delta_{H}.$ It is easy to see that in both cases $\lambda_{H}\circ \lambda_{H} =id_{H}$.

\begin{definition}
Let $D$ be a Hopf algebra. An algebra $B$  is said to be a left $D$-module algebra if $(B, \phi_{B})$ is a left $D$-module and $\eta_{B}$, $\mu_{B}$ are morphisms of left $D$-modules, i.e.,
$$
		\phi_{B}\circ (D\otimes \eta_{B})=\varepsilon_{D}\otimes \eta_{B},\;\;\phi_{B}\circ (D\otimes \mu_{B})=\mu_{B}\circ \phi_{B\otimes B},
$$
	where  $\phi_{B\otimes B}=(\phi_{B}\otimes \phi_{B})\circ (D\otimes c_{D,B}\otimes B)\circ (\delta_{D}\otimes B\otimes B)$ is the left action on $B\otimes B$. 
\end{definition}

\section{Modules over invertible $1$-cocycles and Hopf braces}
We begin the main section of this paper by defining the notion of invertible 1-cocycle in the braided monoidal category ${\sf C}$.  This definition can be directly generalized from the one given for the symmetric setting, for example in the category of vector spaces over a field  ${\mathbb F}$ (see \cite{AGV}).

\begin{definition}
Let $A$, $H$ be  Hopf algebras in {\sf C}. Let's assume that $H$ is a left $A$-module algebra with action $\phi_{H}$. Let $\pi:A\rightarrow H$ be a coalgebra morphism. We will say that $\pi$ is an invertible 1-cocycle if it is an isomorphism such that 
\begin{equation}
\label{1-c}
\pi\circ \mu_{A}=\mu_{H}\circ (\pi \otimes \phi_{H})\circ (\delta_{A}\otimes \pi)
\end{equation}
holds.	
	
Let $\pi:A\rightarrow H$ and $\tau:B\rightarrow D$ be invertible 1-cocycles. A morphism between them is a pair 

$$(f,g)\;:\; \begin{array}{cc} \;\;\;\;A\\ \pi\; \downarrow \\ \;\;\; \;H \end{array}\;\;\;\longrightarrow \; \begin{array}{cc} \;\;\;\;\;\;B\\ \tau\; \downarrow \\ \;\;\; \;\;\;D \end{array}$$
where $f:A\rightarrow B$ and $g:H\rightarrow D$ are algebra-coalgebra morphisms satisfying the following identities:
\begin{equation}
\label{1-c2}
g\circ \pi=\tau\circ f,
\end{equation}
\begin{equation}
\label{1-c3}
g\circ \phi_{H}=\phi_{D}\circ (f\otimes g).
\end{equation}

Then, with these morphisms, invertible 1-cocycles form a category denoted by {\sf IC}. Note that 
$
\pi\circ \eta_{A}=\eta_{H}
$
holds (see \cite{AGV}). 

\end{definition}

\begin{remark}
\label{remH}
It is easy to see that there exists a  functorial connection between the categories {\sf Hopf} and {\sf IC} given by the following: If $A$ is a Hopf algebra, $(A,t_{A}=\varepsilon_{A}\otimes A)$ is a left $A$-module algebra. Then, $id_{A}:A\rightarrow A$ is an object in {\sf IC}. On the other hand, if $f:A\rightarrow B$ is a morphism of Hopf algebras, we have that the pair $(f,f)$ is a morphism in {\sf IC} between $id_{A}:A\rightarrow A$  and $id_{B}:B\rightarrow B$. Therefore, there exists a functor $${\sf H}:{\sf Hopf}\rightarrow {\sf IC}$$ defined on objects by 
$${\sf H}(A)= \begin{array}{cc} \;\;\;\;\;\;\;A\\ id_{A}\; \downarrow \\ \;\;\; \;\;\;\;A \end{array},$$
where the action is $t_{A}=\varepsilon_{A}\otimes A$ (the trivial action), and on morphisms by ${\sf H}(f)=(f,f)$.
\end{remark}
 
As was pointed in \cite{AGV} there exists a closed relation between the Hopf theoretical generalization of skew braces, called Hopf braces, and invertible 1-cocycles in the category of vector spaces over a field 
${\Bbb F}$. In the braided setting we have the same relation and the definition of Hopf brace is the following:

\begin{definition}
\label{H-brace}
{\rm Let $H=(H, \varepsilon_{H}, \delta_{H})$ be a coalgebra in {\sf C}. Let's assume that there are two algebra structures $(H, \eta_{H}^1, \mu_{H}^1)$, $(H, \eta_{H}^2, \mu_{H}^2)$ defined on $H$ and suppose that there exist two endomorphisms of $H$ denoted by $\lambda_{H}^{1}$ and $\lambda_{H}^{2}$. We will say that 
$$(H, \eta_{H}^{1}, \mu_{H}^{1}, \eta_{H}^{2}, \mu_{H}^{2}, \varepsilon_{H}, \delta_{H}, \lambda_{H}^{1}, \lambda_{H}^{2})$$
is a Hopf brace in {\sf C} if:
\begin{itemize}
\item[(i)] $H_{1}=(H, \eta_{H}^{1}, \mu_{H}^{1},  \varepsilon_{H}, \delta_{H}, \lambda_{H}^{1})$ is a Hopf algebra in {\sf C}.
\item[(ii)] $H_{2}=(H, \eta_{H}^{2}, \mu_{H}^{2},  \varepsilon_{H}, \delta_{H}, \lambda_{H}^{2})$ is Hopf algebra in {\sf C}.
\item[(iii)] The  following equality holds:
$$\mu_{H}^{2}\co (H\ot \mu_{H}^{1})=\mu_{H}^{1}\co (\mu_{H}^{2}\ot \Gamma_{H_{1}} )\co (H\ot c_{H,H}\ot H)\co (\delta_{H}\ot H\ot H),$$
\end{itemize}
where  $$\Gamma_{H_{1}}=\mu_{H}^{1}\co (\lambda_{H}^{1}\ot \mu_{H}^{2})\co (\delta_{H}\ot H).$$

Following \cite{RGON}, a Hopf brace will be denoted by ${\mathbb H}=(H_{1}, H_{2})$ or  in a simpler way by ${\mathbb H}$.

If  ${\mathbb H}$ is a Hopf brace in {\sf C}, we will say that ${\mathbb H}$ is cocommutative if $\delta_{H}=c_{H,H}\circ \delta_{H}$, i.e., $H_{1}$ and $H_{2}$ are cocommutative Hopf algebras in {\sf C}.

}
\end{definition}

The previous definition is the general notion of Hopf brace in a braided monoidal setting. If we restrict it to a category of Yetter-Drinfeld modules over a Hopf algebra which antipode is an isomorphism we obtain the definition of braided Hopf brace introduced by H. Zhu and Z. Ying in \cite[Definition 2.1]{Zhu2}.

\begin{definition}
{\rm  Given two Hopf braces ${\mathbb H}$  and  ${\mathbb B}$ in {\sf C}, a morphism $x$ in {\sf C} between the two underlying objects is called a morphism of Hopf braces if both $x:H_{1}\rightarrow B_{1}$ and $x:H_{2}\rightarrow B_{2}$ are algebra-coalgebra morphisms, i.e., Hopf algebra morphisms.
		
Hopf braces together with morphisms of Hopf braces form a category which we denote by {\sf HBr}. This category is a subcategory of the category of Hopf trusses introduced by T. Brzez\'niski in \cite{BRZ1}. 
		
}
\end{definition}

\begin{theorem}
	There exists a functor between the categories {\sf Hopf} and {\sf HBr}.
\end{theorem}

\begin{proof}
	If $H$ is a Hopf algebra, ${\mathbb H}_{triv}=(H, \eta_{H}, \mu_{H}, \eta_{H},\mu_{H}, \varepsilon_{H}, \delta_{H}, \lambda_{H}, \lambda_{H})$ is an object in {\sf HBr}. On the other hand, if $x:H\rightarrow B$ is a morphism of Hopf algebras, we have that the pair $(x,x)$ is a morphism in {\sf HBr} between ${\mathbb H}_{triv}$  and ${\mathbb B}_{triv}$. Therefore, there exists a functor $${\sf H}^{\prime}:{\sf Hopf}\rightarrow {\sf HBr}$$ defined on objects by ${\sf H}^{\prime}(H)={\mathbb H}_{triv}$ and on morphisms by ${\sf H}^{\prime}(x)=(x,x)$.
\end{proof}

Let ${\mathbb H}$  be a Hopf brace in {\sf C}. Then 
$$
\eta_{H}^{1}=\eta_{H}^2, 
$$
holds and, by \cite[Lemma 1.7]{AGV},  in this braided setting  the equality
\begin{equation}
\label{agv1}
\Gamma_{H_{1}}\circ (H\otimes \lambda_{H}^1)=\mu_{H}^{1}\circ ((\lambda_{H}^1\circ \mu_{H}^{2})\otimes H)\circ (H\otimes c_{H,H}) \circ (\delta_{H}\otimes H)
\end{equation}
also holds. Moreover, in our braided context \cite[Lemma 1.8]{AGV} and \cite[Remark 1.9]{AGV} hold and then we have that $(H,\eta_{H}^{1}, \mu_{H}^{1})$ is a left $H_{2}$-module algebra with action $\Gamma_{H_{1}}$ and $\mu_{H}^2$ admits the following expression:
\begin{equation}
\label{eb2}
\mu_{H}^2=\mu_{H}^{1}\circ (H\otimes \Gamma_{H_{1}})\circ (\delta_{H}\otimes H). 
\end{equation}

Now, taking into account that every Hopf brace is an example of Hopf truss, by \cite[Theorem 6.4]{BRZ1}, we have that $(H,\eta_{H}^{1}, \mu_{H}^{1})$ also is a left $H_{2}$-module algebra with action 
$$\Gamma_{H_{1}}^{\prime}=\mu_{H}^{1}\circ (\mu_{H}^{2}\otimes \lambda_{H}^{1})\circ (H\otimes c_{H,H})\circ (\delta_{H}\otimes H)$$
because the symmetry  is not needed in the proof as in the case of $\Gamma_{H_{1}}$.

Finally, by the naturality of $c$ and the coassociativity of $\delta_{H}$, we obtain that 
$$
\mu_{H}^{1}\co (\mu_{H}^{2}\ot \Gamma_{H_{1}} )\co (H\ot c_{H,H}\ot H)\co (\delta_{H}\ot H\ot H)
$$
$$=\mu_{H}^{1}\co (\Gamma_{H_{1}}^{\prime}\otimes  \mu_{H}^{2})\co (H\ot c_{H,H}\ot H)\co (\delta_{H}\ot H\ot H)$$
and then (iii) of Definition \ref{H-brace} is equivalent to 
$$
\mu_{H}^{2}\co (H\otimes \mu_{H}^{1})
=\mu_{H}^{1}\co (\Gamma_{H_{1}}^{\prime}\otimes  \mu_{H}^{2})\co (H\ot c_{H,H}\ot H)\co (\delta_{H}\ot H\ot H).
$$
Therefore, the equality 
\begin{equation}
\label{iii-eq1}
\mu_{H}^{2}
=\mu_{H}^{1}\co (\Gamma_{H_{1}}^{\prime}\otimes  H)\co (H\ot c_{H,H})\co (\delta_{H}\ot H)
\end{equation}
holds.

\begin{remark}
\label{dphi}
Note that by \cite[Lemma 2.5, Lemma 2.7]{Proj23}, if ${\mathbb H}$  is a cocommutative  Hopf brace in {\sf C}, we obtain that the morphisms  $\Gamma_{H_{1}}$ and $\Gamma_{H_{1}}^{\prime}$ are coalgebra morphisms.
\end{remark}

As was proved in \cite[Theorem 1.12]{AGV}, the category of invertible 1-cocycles associated to  a fixed  Hopf algebra is equivalent to the category of Hopf braces  where the first Hopf algebra structure is the same one fixed before. This categorical equivalence remains valid for general invertible 1-cocycles and Hopf braces in braided monoidal categories.

\begin{theorem}
\label{1-th}
The categories {\sf IC} and {\sf HBr} are equivalent.
\end{theorem}

\begin{proof}
The proof follows as in \cite{GONROD} for Hopf trusses. In the following lines, we give a brief summary of this proof to introduce some notation  and  for the convenience of the reader.

Let ${\mathbb H}$ be an object in {\sf HBr}. Then, $id_{H}:H_{2}\rightarrow H_{1}$ is an invertible 1-cocycle. Also, if ${\mathbb H}$ and 
${\mathbb H}^{\prime}$ are objects in {\sf HBr} and  $x:{\mathbb H}\rightarrow {\mathbb H}^{\prime}$ is a morphism between them, the pair $(x,x)$ is a morphism in {\sf IC} between $id_{H}:H_{2}\rightarrow H_{1}$ and $id_{H^{\prime}}:H^{\prime}_{2}\rightarrow H^{\prime}_{1}$. Therefore, there exists a functor ${\sf E}:\;${\sf HBr}$ \rightarrow $ {\sf IC} defined on objects by $${\sf E}({\mathbb H})= \begin{array}{cc} \;\;\;\;\;\;\;H_2\\ id_{H}\; \downarrow \\ \;\;\; \;\;\;\;\;H_1 \end{array}, $$ where $\phi_{H_{1}}=\Gamma_{H_{1}}$, and on morphisms by ${\sf E}(x)=(x,x).$

Conversely,  let  $\pi:A\rightarrow H$ be an object in {\sf IC}. Define $\mu_{H_{\pi}}:=\pi\circ\mu_{A}\circ (\pi^{-1}\otimes \pi^{-1})$,  $\eta_{H_{\pi}}:=\eta_{H}$ and $\lambda_{H_{\pi}}=\pi\circ \lambda_{A}\circ \pi^{-1}$. Then, 
if we denote by $H_{\pi}$ the algebra $(H, \eta_{H_{\pi}},\mu_{H_{\pi}})$, $$(H,\eta_{H}, \mu_{H}, \eta_{H_{\pi}},\mu_{H_{\pi}}, \varepsilon_{H}, \delta_{H}, \lambda_{H}, \lambda_{H_{\pi}})$$  is an object in {\sf HBr} that we will denote by 
${\mathbb H}_{\pi}=(H, H_{\pi})$. Moreover, if $(f,g)$ is a morphism in {\sf IC} between $\pi:A\rightarrow H$ and $\pi^{\prime}:A^{\prime}\rightarrow H^{\prime}$, the  morphism $g$ is a morphism in {\sf HBr} between ${\mathbb H}_{\pi}$ and ${\mathbb H}^{\prime}_{_{\pi^{\prime}}}$. As a consequence of these facts, we have a functor ${\sf Q}:\;{\sf IC} \rightarrow  {\sf HBr}$ defined by 
$${\sf Q}(\begin{array}{cc} \;\;\;\;A\\ \pi\; \downarrow \\ \;\;\; \;H \end{array})={\mathbb H}_{\pi}$$  on objects and by 
${\sf Q}((f,g))=g$ on morphisms. 

The functors induce an equivalence between the two categories because, clearly, ${\sf QE}={\sf id}_{{\sf HBr}}$ and, on the other hand,  ${\sf EQ}\backsimeq {\sf id}_{{\sf IC}}$ because, if $\Gamma_{H}=\mu_{H}\co (\lambda_{H}\ot \mu_{H_{\pi}})\co (\delta_{H}\ot H)$, 
\begin{equation}
\label{gamphi}
\phi_{H}=\Gamma_{H}\circ (\pi\otimes H)
\end{equation}
holds and 
$$(\pi, id_{H}):\begin{array}{cc} \;\;\;\;A\\ \pi\; \downarrow \\ \;\;\; \;H \end{array} \;\;\;\longrightarrow \;  \begin{array}{cc} \;\;\;\;\;\;\;H_{\pi}\\ id_{H}\; \downarrow \\ \;\;\; \;\;\;\;\; H \end{array}={\sf EQ}(\begin{array}{cc} \;\;\;\;A\\ \pi\; \downarrow \\ \;\;\; \;H \end{array})$$
is an isomorphism in {\sf IC}.

\end{proof}

\begin{lemma}
Let $\pi:A\rightarrow H$ be an object in {\sf IC} with action $\phi_H$. Then 
$$
\phi_H\circ (A\ot (\lambda_H\circ\pi))
$$
$$ = \mu_H\circ ((\lambda_H\circ \mu_H)\ot H)\circ  (\pi\ot\phi_H\ot \pi)\circ (\delta_A\ot H\ot A)\circ (A\ot c_{A,H})\circ (\delta_A\ot \pi).$$
\end{lemma}

\begin{proof} The proof is the following: 

\begin{itemize}
\item[ ]$\hspace{0.38cm}\mu_H\circ ((\lambda_H\circ \mu_H)\ot H)\circ  (\pi\ot\phi_H\ot \pi)\circ (\delta_A\ot H\ot A)\circ (A\ot c_{A,H})\circ (\delta_A\ot \pi)$
\item [ ]$=\mu_{H}\circ ((\lambda_{H}\circ \pi\circ \mu_{A})\otimes \pi)\circ (A\otimes c_{A,A})\circ (\delta_{A}\otimes A) $ {\scriptsize (by naturality of $c$ and (\ref{1-c}))}
\item [ ]$=\mu_{H}\circ ((\lambda_{H}\circ  \mu_{H}^{\pi})\otimes H)\circ (H\otimes c_{H,H})\circ ((\delta_{H}\circ \pi)\otimes \pi) $ {\scriptsize (by naturality of $c$, the condition of coalgebra}
\item[ ]$\hspace{0.38cm}${\scriptsize  isomorphism for $\pi$ and the definition of $\mu_{H}^{\pi}$)}
\item [ ]$=\Gamma_{H}\circ (\pi\otimes (\lambda_{H}\circ \pi)) $ {\scriptsize (by (\ref{agv1}) for ${\mathbb H}_{\pi}$)}
\item [ ]$= \phi_{H}\circ (A\otimes (\lambda_{H}\circ \pi))$ {\scriptsize (by (\ref{gamphi})).}
\end{itemize}	
\end{proof}

\begin{theorem}
\label{phi-bar-phi}
Let $A$ and $H$ be Hopf algebras in {\sf C}, let $\pi:A\to H$ be an isomorphism of coalgebras such that $\pi\circ \eta_A = \eta_H$, and let's assume that $H$ is a left $A$-module algebra with action $\phi_H:A\ot H\to H$. Then the following are equivalent:
\begin{enumerate}
\item[(i)] The morphism $\pi:A\to H$ is an invertible 1-cocycle.
\item[(ii)] The pair $(H,\phi_H^{\prime}) $ is a left $A$-module algebra, where \[\phi_H^{\prime} = \mu_H\circ (\mu_H\ot H)\circ (\pi\ot \phi_H\ot (\lambda_H\circ \pi))\circ (\delta_A\ot c_{A,H})\circ (\delta_A\ot H),\]
and moreover 
\begin{equation}\label{mu-pi-barphi}
\pi\circ \mu_A = \mu_H\circ (\phi_H^{\prime}\ot \pi)\circ (A\ot c_{A,H})\circ (\delta_A\ot \pi)
\end{equation} 
holds.
\end{enumerate}
\end{theorem}

\begin{proof} First we will prove that (i) $\Rightarrow $ (ii). Indeed, let ${\mathbb H}_{\pi}$ be the Hopf brace introduced in the proof of Theorem \ref{1-th}. Then $(H, \Gamma_{H})$ is a left $H_{\pi}$-module algebra and (\ref{gamphi}) holds. Then, if 
\[\phi_H^{\prime} = \mu_H\circ (\mu_H\ot H)\circ (\pi\ot \phi_H\ot (\lambda_H\circ \pi))\circ (\delta_A\ot c_{A,H})\circ (\delta_A\ot H),\]
we have that 
\begin{equation}
\label{gamphi1}
\phi_{H}^{\prime}=\Gamma_{H}^{\prime}\circ (\pi\otimes H)
\end{equation}
holds because 
\begin{itemize}
\item[ ]$\hspace{0.38cm}\phi_{H}^{\prime}$
\item [ ]$=\mu_{H}\circ ((\mu_{H}\circ (H\otimes (\phi_{H}\circ (\pi^{-1}\otimes H)))\circ (\delta_{H}\otimes H))\otimes \lambda_{H})\circ (H\otimes c_{H,H})\circ ((\delta_{H}\circ \pi)\otimes H) $ {\scriptsize (by }
\item[ ]$\hspace{0.38cm}$ {\scriptsize  the condition of coalgebra isomorphism for $\pi$ and the naturality of $c$)}
\item [ ]$=\mu_{H}\circ ((\mu_{H}\circ (H\otimes \Gamma_{H})\circ (\delta_{H}\otimes H))\otimes \lambda_{H})\circ (H\otimes c_{H,H})\circ ((\delta_{H}\circ \pi)\otimes H) $ {\scriptsize (by (\ref{gamphi}))}
\item [ ]$=\mu_{H}\circ (\mu_{H_{\pi}}\otimes \lambda_{H})\circ (H\otimes c_{H,H})\circ ((\delta_{H}\circ \pi)\otimes H)$  {\scriptsize (by (\ref{eb2}))}
\item [ ]$= \Gamma_{H}^{\prime}\circ (\pi\otimes H)$ {\scriptsize (by definition of $\Gamma_{H}^{\prime}$).}
\end{itemize}

Then, as a consequence of (\ref{gamphi1}), $(H,\phi_{H}^{\prime})$ is a left $A$-module algebra because
$(H,\Gamma_{H}^{\prime})$ is a left $H_{\pi}$-module algebra. 
Finally, 
\begin{itemize}
\item[ ]$\hspace{0.38cm}\mu_H\circ (\phi_H^{\prime}\ot \pi)\circ (A\ot c_{A,H})\circ (\delta_A\ot \pi) $
\item [ ]$=\mu_{H}\circ (\Gamma_{H}^{\prime}\otimes H)\circ (H\otimes c_{H,H})\circ ((\delta_{H}\circ \pi)\otimes \pi) $ {\scriptsize (by (\ref{gamphi1}),  the naturality of $c$ and the condition of coalgebra morphism for $\pi$)}
\item [ ]$= \mu_{H_{\pi}}\circ (\pi\otimes \pi)$ {\scriptsize (by (\ref{iii-eq1}))}
	\item [ ]$=\pi\circ \mu_{A} ${\scriptsize (by definition of $\mu_{H_{\pi}}$)}
\end{itemize}
and then (\ref{mu-pi-barphi}) holds.	

Conversely, to prove that (ii) $\Rightarrow$ (i) we only need to show that (\ref{1-c}) holds. Indeed:
\begin{itemize}
\item[ ]$\hspace{0.38cm}\pi\circ \mu_{A} $ 
\item [ ]$=\mu_H\circ (\phi_H^{\prime}\ot \pi)\circ (A\ot c_{A,H})\circ (\delta_A\ot \pi)$
{\scriptsize (by (\ref{mu-pi-barphi}))}
\item [ ]$= \mu_H\circ ((\mu_H\circ (\mu_H\ot H)\circ (\pi\ot \phi_H\ot (\lambda_H\circ \pi))\circ (\delta_A\ot c_{A,H})\circ (\delta_A\ot H))\ot \pi)\circ (A\ot c_{A,H})\circ (\delta_A\ot \pi)$ 
\item[ ]$\hspace{0.38cm}$ {\scriptsize (by the definition of $\phi_{H}^{\prime}$)}
\item [ ]$=\mu_{H}\circ ( (\mu_{H}\circ (\pi \otimes \phi_{H})\circ (\delta_{A}\otimes H))\otimes ((\lambda_{H}\ast id_{H})\circ \pi))\circ (A\otimes c_{A,H})\circ (\delta_{A}\otimes \pi) ${\scriptsize (by the associativity}
\item[ ]$\hspace{0.38cm}$ {\scriptsize  of $\mu_{H}$, the naturality of $c$, the coassociativity of $\delta_{A}$ and the condition of coalgebra morphism for $\pi$)}
\item [ ]$= \mu_{H}\circ (\pi \otimes \phi_{H})\circ (\delta_{A}\otimes \pi)${\scriptsize (by (\ref{antipode}), naturality of $c$ and the unit and counit properties).}
\end{itemize}	
\end{proof}

\begin{remark}
Observe that in the previous theorem we can recover $\phi_H$ from $\phi_H^{\prime}$ as
\[\phi_H = \mu_H\circ (\mu_H\ot H)\circ ((\lambda_H\circ \pi)\ot \phi_H^{\prime}\ot \pi)\circ (\delta_A\ot c_{A,H})\circ (\delta_A\ot H),\]
and $\phi_H$ induces a left $A$-module algebra structure on $H$ if, and only if, $\phi_H^{\prime}$ does. As a consequence, we can define an invertible 1-cocycle as a coalgebra isomorphism satisfying condition (\ref{mu-pi-barphi}), where $(H, \phi_H^{\prime})$ is an $A$-module algebra. 
\end{remark}

\begin{lemma}
\label{ic-c}
Let $\pi:A\rightarrow H$ be an object in {\sf IC} with action $\phi_H$ and such that $H$ is cocommutative. Then $\phi_{H}$ is a coalgebra morphism.
\end{lemma}
\begin{proof}
	The proof follows directly from the equality (\ref{gamphi}) and Remark \ref{dphi}. 
\end{proof}

\begin{lemma}
Let $\pi:A\rightarrow H$ be an object in {\sf IC} with action $\phi_H$ and such that $H$ is cocommutative. Then the action $\phi_{H}^{\prime}$ defined in Theorem \ref{phi-bar-phi} is a coalgebra morphism.
\end{lemma}
\begin{proof}
	The proof follows directly from the equality (\ref{gamphi1}) and Remark \ref{dphi}. 
\end{proof}

\begin{remark}
Let	${\sf H}:{\sf Hopf}\rightarrow {\sf IC}$ be the functor defined in Remark \ref{remH}. If $A$ is a Hopf algebra, its image by {\sf H} is the invertible 1-cocycle $id_{A}:A\rightarrow A$ where the action is the trivial one, i.e., $\phi_{A}=t_{A}=\varepsilon_{A}\otimes A$ and then  $\phi_{A}$ is a coalgebra morphism.  Also, by (ii) of Theorem \ref{phi-bar-phi}, we have that 
$$\phi_{A}^{\prime}=\mu_{A}\circ (\mu_{A}\otimes \lambda_{A})\circ (A\otimes c_{A,A})\circ ( \delta_{A}\otimes A)=\varphi_{A}^{ad}.$$
\end{remark}

Taking into account the previous considerations, we can introduce the notion of left module for an invertible 1-cocycle.

\begin{definition}
\label{def-pi-module}
Let $\pi:A\to H$ be an invertible 1-cocycle. A left module over the invertible 1-cocycle $\pi:A\to H$ is a 6-tuple \((M,N, \phi_M, \varphi_M,  \phi_N, \gamma)\) where
\begin{enumerate}
\item[(i)] $(M, \phi_M)$ is a left $A$-module and $(M, \varphi_M)$ is a left $H$-module such that
\begin{equation}
\label{p-v}
\phi_M\circ (A\ot \varphi_M) = \varphi_M\circ (\phi_H\ot \phi_M)\circ (A\ot c_{A,H}\ot M)\circ (\delta_A\ot H\ot M).
\end{equation}
\item[(ii)] $(N, \phi_N)$ is a left $A$-module.
\item[(iii)] $\gamma:N\to M$ is an isomorphism in $\sf C$ such that 
\begin{equation}\label{eq-gamma}
\gamma\circ \phi_N = \varphi_M\circ (\pi\ot \phi_M)\circ (\delta_A\ot \gamma).
\end{equation}
\end{enumerate} 

Let \((M,N, \phi_M, \varphi_M,  \phi_N, \gamma)\) and 
\((M^{\prime},N^{\prime}, \phi_{M^{\prime}}, \varphi_{M^{\prime}},  \phi_{N^{\prime}}, \gamma^{\prime})\)
be left modules over an invertible 1-cocycle $\pi:A\to H$. A morphism between them is a pair $(h,l)$ such that $h:M\rightarrow M^{\prime}$ is a morphism of  left $A$-modules and left $H$-modules,  $l:N\rightarrow N^{\prime}$ is a morphism of  left $A$-modules and the following identity holds:
\begin{equation}\label{fg-g}
h\circ \gamma = \gamma^{\prime}\circ l.
\end{equation}

Note that, by (\ref{fg-g}), the morphism $l$ is determined by $h$ because $l=(\gamma^{\prime})^{-1}\circ h\circ \gamma.$

With the obvious composition of morphisms, left modules over an invertible 1-cocycle $\pi:A\to H$ with action $\phi_{H}$ form a category that we will denote by $\;_{(\pi, \phi_{H})}${\sf Mod}.
\end{definition}

\begin{remark}
If  \((M,N, \phi_M, \varphi_M,  \phi_N, \gamma)\) is a left module over the invertible 1-cocycle $\pi:A\to H$, by (\ref{eq-gamma}), we obtain that $\phi_{N}$ is determined by $\phi_{M}$ and $\varphi_{M}$ because
\begin{equation}
\label{req-g1}
 \phi_N =\gamma^{-1}\circ  \varphi_M\circ (\pi\ot \phi_M)\circ (\delta_A\ot \gamma).
\end{equation}

Also, composing in both sides of the equality  (\ref{eq-gamma}) with $(((\lambda_H\circ \pi)\ot A)\circ \delta_A)\ot \gamma^{-1}$ on the right and with $\varphi_M$ on the left we obtain the identity
\begin{equation}
\label{req-g2}
\phi_M =\varphi_{M}\circ ((\lambda_{H}\circ \pi)\otimes (\gamma\circ \phi_{N}))\circ (\delta_{A}\otimes \gamma^{-1}).
\end{equation}

\end{remark}

\begin{example}
\label{exmod1}
It is easy to see that if $\pi:A\to H$ is an invertible 1-cocycle, the 6-tuple \((H,A, \phi_H, \mu_H,  \mu_A, \pi)\) is an example of left module over the invertible cocycle $\pi:A\to H$. 

Also, the unit object $K$ of {\sf C} is  an example of left module over the invertible 1-cocycle $\pi:A\to H$, where $\phi_{K}=\varepsilon_{A}$, $\varphi_{K}=\varepsilon_{H}$ and $\gamma=id_{K}$ because by (\ref{gamphi}) we have that $\varepsilon_{H}\circ \phi_{H}=\varepsilon_{A}\otimes \varepsilon_{H}.$ We call $(K,K, \phi_K , \varphi_K, \phi_K , id_K)$ the trivial left module over the invertible 1-cocycle $\pi:A\to H$.

Note that,  if $(M,\phi_{M})$ is an object in $_{\sf H}${\sf Mod}, the 6-tuple $(M,M, \phi_{M}, t_{M}=\varepsilon_{H}\otimes M, \phi_{M}, id_{M})$ is a left module over the invertible 1-cocycle $id_{H}:H\rightarrow H$ defined in Remark \ref{remH}. Also, if $f$ is a morphism between two left $H$-modules $(M,\phi_{M})$ and $(P,\phi_{P})$, the pair $(f,f)$ is a morphism of left modules over the invertible 1-cocycle $id_{H}:H\rightarrow H$ between $(M,M, \phi_{M}, t_{M}, \phi_{M}, id_{M})$ and $(P,P, \phi_{P}, t_{P}, \phi_{P}, id_{P})$. Therefore, we have a functor $${\sf I}_{H}:\;_{\sf H}{\sf Mod}\;\rightarrow \;_{(id_{H}, t_{H})}{\sf Mod}$$ defined on objects by 
$${\sf I}_{H}((M,\phi_{M}))= (M,M, \phi_{M}, t_{M}, \phi_{M}, id_{M})$$
and on morphisms by ${\sf I}_{H}(f)=(f,f)$.

\end{example}

\begin{theorem}
\label{modfg}
Assume that  $(f,g)$ is a morphism  between the invertible 1-cocycles $\pi:A\rightarrow H$ and $\tau:B\rightarrow D$. Then, there exists a functor 
$${\sf M}_{(f,g)}:\;_{(\tau, \phi_{D})}{\sf Mod}\;\rightarrow \;_{(\pi, \phi_{H})}{\sf Mod}$$
defined on objects by 
$${\sf M}_{(f,g)}((P,Q, \phi_P, \varphi_P,  \phi_Q, \theta))=(P,Q, \phi_P^{\pi}=\phi_P\circ (f\otimes P), \varphi_P^{\pi}=\varphi_P\circ (g\otimes P),  \phi_Q^{\pi}=\phi_Q\circ (f\otimes Q), \theta)$$
and on morphisms by the identity.
\end{theorem}

\begin{proof} The existence of the functor ${\sf M}_{(f,g)}$ is a consequence of the following facts: Trivially $(P, \phi_P^{\pi})$, $(Q, \phi_Q^{\pi})$ are left $A$-modules and $(P, \varphi_P^{\pi})$ is a left $H$-module. Also,  
\begin{itemize}
\item[ ]$\hspace{0.38cm}\varphi_P^{\pi}\circ (\phi_H\ot \phi_P^{\pi})\circ (A\ot c_{A,H}\ot P)\circ (\delta_A\ot H\ot P) $
\item [ ]$=\varphi_P\circ ((g\circ \phi_{H})\otimes (\phi_{P}\circ (f\otimes P)))\circ (A\otimes c_{A,H}\otimes P)\circ (\delta_{A}\otimes H\otimes P )$ {\scriptsize (by definition of $\phi_P^{\pi}$ and $\varphi_P^{\pi}$)}
\item [ ]$=\varphi_P\circ ( \phi_{D}\otimes \phi_{P})\circ (B\otimes c_{B,D}\otimes P)\circ (((f\otimes f)\circ \delta_{A})\otimes g\otimes P ) $ {\scriptsize (by (\ref{1-c3}) and natirality of $c$)}
\item [ ]$=\varphi_P\circ ( \phi_{D}\otimes \phi_{P})\circ (B\otimes c_{B,D}\otimes P)\circ ((\delta_{B}\circ f)\otimes g\otimes P )$ {\scriptsize (by the coalgebra morphism condition for $f$)}
\item [ ]$=\phi_P^{\pi}\circ (A\ot \varphi_P^{\pi}) ${\scriptsize (by (\ref{p-v}))},
\end{itemize}
and
\begin{itemize}
	\item[ ]$\hspace{0.38cm}\varphi_P^{\pi}\circ (\pi\ot \phi_P^{\pi})\circ (\delta_A\ot \theta) $
	\item [ ]$= \varphi_P\circ ((g\circ \pi)\otimes (\phi_{P}\circ (f\otimes P)))\circ (\delta_{A}\otimes \theta)$ {\scriptsize (by definition of $\phi_P^{\pi}$ and $\varphi_P^{\pi}$)}
	\item [ ]$=\varphi_P\circ ((\tau\circ f)\otimes (\phi_{P}\circ (f\otimes P)))\circ (\delta_{A}\otimes \theta)$ {\scriptsize (by (\ref{1-c2}))}
	\item [ ]$= \varphi_P\circ (\tau\otimes \phi_{P})\circ ((\delta_{B}\circ f)\otimes \theta)$ {\scriptsize (by the coalgebra morphism condition for $f$)}
	\item [ ]$=\theta\circ \phi_Q^{\pi}$ {\scriptsize (by (\ref{eq-gamma}))}
\end{itemize}
Then $(P,Q, \phi_P^{\pi}, \varphi_P^{\pi}, \phi_Q^{\pi}, \theta)$ is an object in $\;_{(\pi, \phi_{H})}{\sf Mod}$. Finally, it is obvious that if $(h,l)$ is a morphism in $\;_{(\tau, \phi_{D})}{\sf Mod}$, $(h,l)$ is a morphism in $\;_{(\pi, \phi_{H})}{\sf Mod}$. 

\end{proof}

\begin{remark}
Let $f:H\rightarrow H^{\prime}$ be a Hopf algebra morphisms. Then, by Example \ref{exmod1} and Theorem \ref{modfg}, we have the following commutative diagram 

$$
\setlength{\unitlength}{4mm}
\begin{picture}(14,7.5)
		
\put(4,6){\vector(1,0){5}}
\put(2,5){\vector(0,-1){2}}
\put(11.5,5){\vector(0,-1){2}}
\put(4.2,2){\vector(1,0){5}}
		
\put(6.5,7){\makebox(0,0){$ {\sf I}_{H^{\prime}}$}}
\put(2,6){\makebox(0,0){$\;_{{\sf H}^{\prime}}{\sf Mod}$ }}
\put(12,6){\makebox(0,0){$\;_{(id_{H^{\prime}}, t_{H^{\prime}})}{\sf Mod}$}}
		
\put(1,4.2){\makebox(0,0){${\sf M}_{f}$}}
\put(13,4.2){\makebox(0,0){${\sf M}_{(f,f)}$}}
		
\put(2,2){\makebox(0,0){$\;_{\sf H}{\sf Mod}$}}
\put(12,2){\makebox(0,0){$\;_{(id_{H}, t_{H})}{\sf Mod}$}}
\put(6.5,1){\makebox(0,0){${\sf I}_{H}$}}
		
\end{picture}
$$
where ${\sf M}_{f}$ is the restriction of scalars functor.
\end{remark}

\begin{remark}
If $(f,g)$ is an isomorphism defined between the invertible 1-cocycles $\pi:A\rightarrow H$ and $\tau:B\rightarrow D$ with inverse $(f^{-1}, g^{-1})$, the functor ${\sf M}_{(f,g)}$ is an isomorphism of categories with inverse ${\sf M}_{(f^{-1},g^{-1})}$.  For example, in the proof of Theorem \ref{1-th} we proved that, for all invertible 1-cocycle $\pi:A\rightarrow H$, $(\pi, id_{H})$ is an isomorphism between the invertible 1-cocycles $\pi:A\rightarrow H$ and $id_{H}:H_{\pi}\rightarrow H$. Therefore, the functor 
$${\sf M}_{(\pi,id_{H})}:\;_{(id_{H}, \Gamma_{H})}{\sf Mod}\;\rightarrow \;_{(\pi, \phi_{H})}{\sf Mod}$$
is an isomorphism of categories with inverse 
$${\sf M}_{(\pi^{-1},id_{H})}:\;_{(\pi, \phi_{H})}{\sf Mod} \;\rightarrow \;_{(id_{H}, \Gamma_{H})}{\sf Mod}.$$
\end{remark}

\begin{theorem}
\label{th-mon-cat}
Let's assume that {\sf C} is symmetric with natural isomorphism of symmetry $c$. Let $A$ and $H$ be cocommutative Hopf algebras in {\sf C}. Then the category of left modules over an invertible 1-cocycle $\pi:A\to H$ is symmetric monoidal with unit object the trivial left module over the invertible 1-cocycle $\pi:A\to H$.
\end{theorem}

\begin{proof} Let \((M,N, \phi_M, \varphi_M,  \phi_N, \gamma)\), \((P,Q, \phi_P, \varphi_P,  \phi_Q, \theta)\) be objects in $\;_{(\pi, \phi_{H})}{\sf Mod}$. Then we will define their tensor product as 
$$(M,N, \phi_M, \varphi_M,  \phi_N, \gamma)\otimes (P,Q, \phi_P, \varphi_P,  \phi_Q, \theta)$$
$$=(M\otimes P,N\otimes Q, \phi_{M\otimes P}, \varphi_{M\otimes P},  \phi_{N\otimes Q}, \gamma\otimes \theta)$$
where the left $A$-actions are defined by  $\phi_{M\otimes P}=(\phi_{M}\otimes \phi_{P})\circ (A\otimes c_{A,M}\otimes P)\circ (\delta_{A}\otimes M\otimes P)$, $\phi_{N\otimes Q}=(\phi_{N}\otimes \phi_{Q})\circ (A\otimes c_{A,N}\otimes Q)\circ (\delta_{A}\otimes N\otimes Q)$ and the left $H$-action is $\varphi_{M\otimes P}=(\varphi_{M}\otimes \varphi_{P})\circ (H\otimes c_{H,M}\otimes P)\circ (\delta_{H}\otimes M\otimes P)$. By the monoidal property of the category of modules over a Hopf algebra we have that $(M\otimes P, \phi_{M\otimes P})$ and $(N\otimes Q, \phi_{N\otimes Q})$ are left $A$-modules and $(M\otimes P, \varphi_{M\otimes P})$ is a left $H$-module. Moreover, the equality (\ref{p-v}) holds because
\begin{itemize}
	\item[ ]$\hspace{0.38cm}\varphi_{M\otimes P}\circ (\phi_{H}\otimes \phi_{M\otimes P})\circ (A\otimes c_{A,H}\otimes M\otimes P)\circ (\delta_{A}\otimes H\otimes M\otimes P) $ 
	\item [ ]$=(\varphi_{M}\otimes \varphi_{P})\circ (H\otimes c_{H,M}\otimes P)\circ ((\delta_{H}\circ \phi_{H})\otimes \phi_{M\otimes P})\circ (A\otimes c_{A,H}\otimes M\otimes P)\circ (\delta_{A}\otimes H\otimes M\otimes P) $
	\item[ ]$\hspace{0.38cm}$ {\scriptsize (by definition)}
	\item [ ]$=((\varphi_{M}\circ (H\otimes \phi_{M}))\otimes \varphi_{P})\circ (H\otimes ((A\otimes c_{H,M})\circ (c_{H,A}\otimes M))\otimes A\otimes P)\circ (((\phi_{H}\otimes \phi_{H})\circ \delta_{A\otimes H})$
	\item[ ]$\hspace{0.38cm}\otimes A\otimes M\otimes \phi_{P})\circ (A\otimes ((c_{A,H}\otimes c_{A,M})\circ (A\otimes c_{A,H}\otimes M))\otimes P) \circ (((A\otimes \delta_{A})\circ \delta_{A})\otimes H\otimes M\otimes P)$ 
	\item[ ]$\hspace{0.38cm}$ {\scriptsize (by the naturality of $c$ and the condition of coalgebra morphism for $\phi_{H}$ (see Lemma \ref{ic-c}))}
	\item [ ]$= ((\varphi_{M}\circ (\phi_{H}\otimes \phi_{M}))\otimes (\varphi_{P}\circ (\phi_{H}\otimes \phi_{P}))) \circ  (A\otimes ((H\otimes A\otimes c_{A,M}\otimes H\otimes A)\circ (H\otimes c_{A,A}\otimes c_{H,M}\otimes A)$
	\item[ ]$\hspace{0.38cm}\circ (c_{A,H}\otimes A\otimes H\otimes c_{A,M} ))\otimes  P)\circ (\delta_{A}\otimes ((c_{A,H}\otimes c_{A,H})\circ \delta_{A\otimes H}))\otimes M\otimes P)\circ (\delta_{A}\otimes H\otimes M\otimes P)$  
	\item[ ]$\hspace{0.38cm}$ {\scriptsize  (by the naturality of $c$ and $c_{H,A}\circ c_{A,H}=id_{A\otimes H}$)}
	\item [ ]$=((\varphi_{M}\circ (\phi_{H}\otimes \phi_{M}))\otimes (\varphi_{P}\circ (\phi_{H}\otimes \phi_{P}))) \circ (A\otimes ((c_{A,H}\otimes c_{A,M}\otimes H\otimes A)\circ (A\otimes c_{A,H}\otimes c_{H,M}\otimes A)$
	\item[ ]$\hspace{0.38cm}\circ ((c_{A,A}\circ \delta_{A})\otimes H\otimes H\otimes c_{A,M})\circ (A\otimes H\otimes c_{A,H}\otimes M )\circ (\delta_{A\otimes H}\otimes M))\otimes P)\circ (\delta_{A}\otimes H\otimes M\otimes P)${\scriptsize (by }
	\item[ ]$\hspace{0.38cm}$ {\scriptsize  the coassociativity of $\delta_{A}$ and the naturality of $c$)}
	\item [ ]$=((\varphi_{M}\circ (\phi_{H}\otimes \phi_{M}))\otimes (\varphi_{P}\circ (\phi_{H}\otimes \phi_{P})))\circ (A\otimes ((c_{A,H}\otimes c_{A,M}\otimes c_{A,H})\circ (A\otimes c_{A,H}\otimes c_{A,M}\otimes H)$
	\item[ ]$\hspace{0.38cm}\circ (A\otimes A\otimes c_{A,H}\otimes c_{H,M}))\otimes P)\circ (((\delta_{A}\otimes \delta_{A})\circ \delta_{A})\otimes \delta_{H}\otimes M\otimes P)$
	{\scriptsize (by the cocommutativity, the }
	\item[ ]$\hspace{0.38cm}$ {\scriptsize  coassociativity of $\delta_{A}$ and the naturality of $c$)}
	\item [ ]$= ((\varphi_M\circ (\phi_H\ot \phi_M)\circ (A\ot c_{A,H}\ot M)\circ (\delta_A\ot H\ot M))\otimes (\varphi_P\circ (\phi_H\ot \phi_P)\circ (A\ot c_{A,H}\ot P)$
	\item[ ]$\hspace{0.38cm}\circ (\delta_A\ot H\ot P)) )\circ (A\otimes H\otimes c_{A,M}\otimes H\otimes P)\circ (A\otimes c_{A,H}\otimes c_{H,M}\otimes P)\circ (\delta_{A}\otimes \delta_{H}\otimes M\otimes P)$ 
	\item[ ]$\hspace{0.38cm}$ {\scriptsize (by the naturality of $c$)}
	\item [ ]$=((\phi_M\circ (A\ot \varphi_M))\otimes (\phi_P\circ (A\ot \varphi_P) )\circ (A\otimes H\otimes c_{A,M}\otimes H\otimes P)\circ (A\otimes c_{A,H}\otimes c_{H,M}\otimes P)$
	\item[ ]$\hspace{0.38cm}\circ (\delta_{A}\otimes \delta_{H}\otimes M\otimes P)$ {\scriptsize (by (\ref{p-v}))}
	\item [ ]$=\phi_{M\otimes P}\circ (A\otimes \varphi_{M\otimes P}) $ {\scriptsize (by the naturality of $c$)}
    \end{itemize}
and, on the other hand, (\ref{eq-gamma}) follows by 
\begin{itemize}
	\item[ ]$\hspace{0.38cm}\varphi_{M\otimes P}\circ (\pi\ot \phi_{M\otimes P})\circ (\delta_A\ot \gamma\otimes \theta) $
	\item [ ]$= ((\varphi_{M}\circ (\pi\ot (\phi_{M}\circ (A\otimes \gamma))))\otimes (\varphi_{ P}\circ (\pi\otimes  (\phi_{P}\circ (A\otimes \theta)))))\circ (A\otimes ((A\otimes c_{A,N})\circ ((c_{A,A}\circ \delta_{A})$ 
	\item[ ]$\hspace{0.38cm}\otimes N))\otimes A\otimes Q)\circ (\delta_{A}\otimes c_{A,N}\otimes Q)\circ (\delta_{A}\otimes N\otimes Q)$ {\scriptsize (by the coalgebra morphism condition for $\pi$, the }
	\item[ ]$\hspace{0.38cm}$ {\scriptsize  coassociativity of $\delta_{A}$ and the naturality of $c$)}
	\item [ ]$=((\varphi_{M}\circ (\pi\ot (\phi_{M}\circ (A\otimes \gamma))))\otimes (\varphi_{ P}\circ (\pi\otimes  (\phi_{P}\circ (A\otimes \theta)))))\circ (\delta_{A}\otimes ((c_{A,N}\otimes A)\circ (A\otimes c_{A,N})$
	\item[ ]$\hspace{0.38cm}\circ (\delta_{A}\otimes N))\otimes Q)\circ (\delta_{A}\otimes N\otimes Q) $ {\scriptsize (by the cocommutativity and the coassociativity of $\delta_{A}$)}
	\item [ ]$=((\varphi_{M}\circ (\pi\ot (\phi_{M}\circ (A\otimes \gamma))\circ (\delta_{A}\otimes N)))\otimes (\varphi_{ P}\circ (\pi\otimes  (\phi_{P}\circ (A\otimes \theta))\circ (\delta_{A}\otimes Q))))$
	\item[ ]$\hspace{0.38cm}\circ (A\otimes c_{A,N}\otimes Q)\circ (\delta_{A}\otimes N\otimes Q) $ {\scriptsize (by the the naturality of $c$)}
	\item [ ]$=(\gamma\otimes \theta)\circ \phi_{N\otimes Q} $ {\scriptsize (by (\ref{eq-gamma})).}
\end{itemize}

Finally, by the cocommutativity condition and the symmetry condition,  it is easy to prove that $(c_{M,P}, c_{N,Q})$ is a morphism  in $\;_{(\pi, \phi_{H})}{\sf Mod}$ between $(M,N, \phi_M, \varphi_M,  \phi_N, \gamma)\otimes (P,Q, \phi_P, \varphi_P,  \phi_Q, \theta)$ and $(P,Q, \phi_P, \varphi_P,  \phi_Q, \theta)\otimes (M,N, \phi_M, \varphi_M,  \phi_N, \gamma)$. As a consequence, $\;_{(\pi, \phi_{H})}{\sf Mod}$ is symmetric.
\end{proof}

Following \cite{RGON} we recall the notion of left module over  a Hopf brace.

\begin{definition}
\label{l-mod}
{\rm Let ${\Bbb H}$ be a Hopf brace. A left ${\Bbb H}$-module is a triple $(M,\psi_{M}^{1}, \psi_{M}^{2})$, where $(M,\psi_{M}^{1})$ is a left $H_{1}$-module, $(M, \psi_{M}^{2})$  is a left $H_{2}$-module and the following identity 
\begin{equation}
\label{mod-l1}
\psi_{M}^{2}\co (H\ot \psi_{M}^{1})=\psi_{M}^1\co (\mu_{H}^{2}\ot \Gamma_{M})\co (H\ot c_{H,H}\ot M)\co (\delta_{H}\ot H\ot M)
\end{equation}
holds, where 
$$\Gamma_{M}=\psi_{M}^{1}\co (\lambda_{H}^1\ot \psi_{M}^{2})\co (\delta_{H}\ot M).$$
		
Given two left ${\Bbb H}$-modules  $(M,\psi_{M}^1, \psi_{M}^{2})$  and  $(N,\psi_{N}^1, \psi_{N}^{2})$, a morphism $f:M\rightarrow N$  is called a morphism of left ${\Bbb H}$-modules if  $f$ is a morphism of left $H_{1}$-modules and left $H_{2}$-modules. Left ${\Bbb H}$-modules  with morphisms of left ${\Bbb H}$-modules  form a category which we denote by $\;_{\Bbb H}${\sf Mod}. 
}
\end{definition}

\begin{example}
Let ${\Bbb H}$ be a Hopf brace. The triple $(H, \mu_{H}^{1}, \mu_{H}^{2})$ is an example of  left ${\Bbb H}$-module. Also, if $K$ is the unit object of ${\sf C}$, $(K, \psi_{K}^1=\varepsilon_{H}, \psi_{K}^2=\varepsilon_{H})$ is a left ${\Bbb H}$-module called the trivial module.
		
Let $H=(H, \eta_{H}, \mu_{H},  \varepsilon_{H}, \delta_{H}, \lambda_{H})$ be a Hopf algebra. Then $(H,\mu_{H}, \mu_{H})$ is an example of left ${\Bbb H}$-module for the Hopf brace ${\Bbb H}$ with $H_{1}=H_{2}=H$. Also, if $(M, \psi_{M})$ is a left $H$-module, the triple $(M, \psi_{M}, \psi_{M})$ is a left ${\Bbb H}$-module for the same Hopf brace. Then, there exists an obvious  functor $J:_{\sf H}${\sf Mod}$\;\rightarrow $ $_{\Bbb H}${\sf Mod} defined on objects by $J((M, \psi_{M}))=(M, \psi_{M}, \psi_{M})$ and by the identity on morphisms. Also, there exists a functor $L:_{\Bbb H}${\sf Mod}$\;\rightarrow $ $_{\sf H}${\sf Mod} defined on objects by $L((M, \psi_{M}^1, \psi_{M}^{2}))=(M, \psi_{M}^1)$ and by the identity on morphisms. Obviously, $L\circ J={\sf id}_{_H{\sf Mod}}.$

\end{example}

\begin{remark}
As was pointed in \cite{RGON}, Definition \ref{l-mod} is weaker than the one introduced by H. Zhu in \cite{Zhu}. For this author, if ${\mathbb H}$ is a Hopf brace, a left ${\Bbb H}$-module is a triple $(M,\psi_{M}^1, \psi_{M}^2)$, where $(M,\psi_{M}^1)$ is a left $H_{1}$-module, $(M, \psi_{M}^2)$  is a left $H_{2}$-module, and the equalities (\ref{mod-l1}) and 
\begin{equation}
\label{ch}
(\psi_{M}^{2}\ot H)\co (H\ot c_{H,M})\co (\delta_{H}\ot M)=(\psi_{M}^1\ot H)\co (H\ot c_{H,M})\co (\delta_{H}\ot \Gamma_{M})\co (\delta_{H}\ot M)
\end{equation}
hold (see \cite[Definition 3.1, Lemma 3.2]{Zhu}). Thus, for an arbitrary Hopf brace ${\Bbb H}$, a left ${\Bbb H}$-module in the sense of Zhu is a left ${\Bbb H}$-module in our sense. Moreover, if ${\Bbb H}$ is cocommutative, (\ref{ch}) hold for any left ${\Bbb H}$-module as in Definition \ref{l-mod}. As a consequence, in the cocommutative setting, \cite[Definition 3.1]{Zhu} and Definition \ref{l-mod}  are equivalent.  Moreover, as was pointed in \cite{Proj23}, if we use  Definition \ref{l-mod}, trivially, $(H,\mu_{H}^1, \mu_{H}^2)$ is a left ${\mathbb H}$-module but, if we work with the definition introduced by Zhu, the condition of left ${\mathbb H}$-module for  $(H,\mu_{H}^1, \mu_{H}^2)$ implies that  the following identity
$$
(\Gamma_{H_1}\otimes H)\circ (H\otimes c_{H,H})\circ (\delta_{H}\otimes H)=(\Gamma_{H_1}\otimes H)\circ (H\otimes c_{H,H})\circ ((c_{H,H}\circ\delta_{H})\otimes H)
$$
holds. Therefore, if {\sf C} is  symmetric, for example the category of vector spaces over a fiel ${\mathbb K}$, $(H, \Gamma_{H_1})$ is in the cocommutativity class of $H$ (see \cite{CCH} for the definition). In other words, under certain circumstances, for example, the lack of cocommutativity, the category of left modules over a Hopf brace introduced by Zhu may not contain the trivial object $(H,\mu_{H}^1, \mu_{H}^2).$

\end{remark}

\begin{remark}
\label{yo} 
It is easy to show that  (\ref{mod-l1}) is equivalent to 
$$
\psi_{M}^{2}\co (H\ot \psi_{M}^{1})=\psi_{M}^1\co (\Gamma_{H_1}^{\prime}\ot \psi_{M}^{2})\co (H\ot c_{H,H}\ot M)\co (\delta_{H}\ot H\ot M).
$$
and, by \cite[Lemma 2.11]{Proj23}, the following equality holds:
\begin{equation}
	\label{GMH1}
	\Gamma_{M}\circ (H\otimes \psi_{M}^1)=\psi_{M}^1\circ (\Gamma_{H_1}\otimes \Gamma_{M})\circ (H\otimes c_{H,H}\otimes M)\circ (\delta_{H}\otimes H\otimes M).
\end{equation}

Moreover,  $(M, \Gamma_{M})$ is a left $H_{2}$-module.

Finally, if {\sf C} is symmetric with natural isomorphism of symmetry $c$ and  ${\mathbb H}$ is a cocommutative Hopf brace in {\sf C},   the category of left modules over ${\mathbb H}$ is symmetric monoidal with unit object the trivial left module over the ${\mathbb H}$ (see \cite[Theorem 2.28]{Proj23}).

\end{remark}

\begin{theorem}
\label{th-mon-cat21}
Let ${\mathbb H}$ be a Hopf brace and let ${\sf E}({\mathbb H})$ be the invertible 1-cocycle induced by the  functor {\sf E} introduced in the proof of Theorem \ref{1-th}. There  exists a functor 
$${\sf G}_{{\mathbb H}}:\;_{{\mathbb H}}{\sf Mod}\;\rightarrow \;_{(id_{H}, \Gamma_{H_1})}{\sf Mod}$$ 
defined on objects by 
$${\sf G}_{{\mathbb H}}((M, \psi_{M}^1, \psi_{M}^2))=
(M,M, \widehat{\phi}_{M}=\Gamma_{M}, \widehat{\varphi}_{M}=\psi_{M}^1,
 \overline{\phi}_{M}=\psi_{M}^2, id_{M})$$
and on morphisms by ${\sf G}_{{\mathbb H}}(f)=(f,f)$.
\end{theorem}

\begin{proof}
By \cite[Lemma 2.11]{Proj23} we know that $(M, \widehat{\phi}_{M}=\Gamma_{M})$ is a left $H_{2}$-module and,  by assumption, $(M, \widehat{\varphi}_{M}=\psi_{M}^1)$ is a left $H_{1}$-module and $(M, \overline{\phi}_{M}=\psi_{M}^2)$ is a left $H_{2}$-module. On the other hand, by (\ref{GMH1}) we have that 
$$\widehat{\phi}_{M}\circ (H\otimes \widehat{\varphi}_{M})=\widehat{\varphi}_{M}\circ (\Gamma_{H_1}\otimes \widehat{\phi}_{M})\circ (H\otimes c_{H,H}\otimes M)\circ (\delta_{H}\otimes H\otimes M)$$
and then, (\ref{p-v}) holds. Also, 
\begin{itemize}
	\item[ ]$\hspace{0.38cm} \widehat{\varphi}_{M}\circ (H\ot \widehat{\phi}_{M})\circ (\delta_H\ot M)$
	\item [ ]$=\psi_{M}^1\circ (H\ot \Gamma_{M})\circ (\delta_H\ot M) $ {\scriptsize (by definition of $ \widehat{\varphi}_{M}$ and $\widehat{\phi}_{M}$)}
	\item [ ]$=\psi_{M}^1\circ (H\ot (\psi_{M}^{1}\co (\lambda_{H}^1\ot \psi_{M}^{2})\co (\delta_{H}\ot M)))\circ (\delta_{H}\otimes M) $ {\scriptsize (by definition of $\Gamma_{M}$)}
	\item [ ]$=\psi_{M}^1\circ ((id_{H}\ast \lambda_{H}^{1})\otimes \psi_{M}^2)\circ (\delta_{H}\otimes M) $ {\scriptsize (by the condition of left $H_{1}$-module of $(M, \psi_{M}^1)$ and the coassociativity}
	\item[ ]$\hspace{0.38cm}$ {\scriptsize  of $\delta_{H}$)}
	\item [ ]$=\overline{\phi}_{M}$ {\scriptsize (by (\ref{antipode}), the counit properties, the condition of left $H_{1}$-module of $(M, \psi_{M}^1)$ and the definition of $ \overline{\phi}_{M}$).}
\end{itemize}

Finally, it is easy to show that if $f$ is a morphism in $\;_{{\mathbb H}}{\sf Mod}$ between the objects $(M, \psi_{M}^1, \psi_{M}^2)$ and $(M^{\prime}, \psi_{M^{\prime}}^1, \psi_{M^{\prime}}^2)$, the pair $(f,f)$ is a morphism in $\;_{(id_{H}, \Gamma_{H_1})}{\sf Mod}$ between ${\sf G}_{{\mathbb H}}((M, \psi_{M}^1, \psi_{M}^2))$ and ${\sf G}_{{\mathbb H}}((M^{\prime}, \psi_{M^{\prime}}^1, \psi_{M^{\prime}}^2))$.
\end{proof}

\begin{theorem}
\label{pHpi}
Let $\pi:A\to H$ be an invertible 1-cocycle. Then the categories $_{(\pi, \phi_{H})}{\sf Mod}$ and $\;_{{\Bbb H}_{\pi}}${\sf Mod} are equivalent.
\end{theorem}

\begin{proof} First of all we will prove that there  exists a functor 
$${\sf H}_{{\sf br}}^{\pi}:\;_{(\pi, \phi_{H})}{\sf Mod}\;\rightarrow \;_{{\mathbb H}_{\pi}}{\sf Mod}$$ 
defined on objects by 
$${\sf H}_{{\sf br}}^{\pi}((M,N, \phi_M, \varphi_M,  \phi_N, \gamma))=
(M,\overline{\psi}_{M}^{1}=\varphi_M, \overline{\psi}_{M}^{2}=\gamma\circ \phi_N\circ (\pi^{-1}\otimes \gamma^{-1}))$$
and on morphisms by ${\sf H}_{{\sf br}}^{\pi}((h,l))=h$. Indeed: By assumption, $(M,\overline{\psi}_{M}^{1}=\varphi_M)$ is a left $H$-module and, using the condition of left $A$-module of $N$, we obtain that $(M, \overline{\psi}_{M}^{2}=\gamma\circ \phi_N\circ (\pi^{-1}\otimes \gamma^{-1}))$ is a left $H_{\pi}$-module. Also, by (\ref{req-g2}), we have that the identity 
\begin{equation}
\label{pHpi1}
\phi_{M}\circ (\pi^{-1}\otimes M)=\overline{\Gamma}_{M}
\end{equation}
holds, where $\overline{\Gamma}_{M}=\overline{\psi}_{M}^{1}\circ (\lambda_{H}\otimes \overline{\psi}_{M}^{2})\circ (\delta_{H}\otimes M)$. Then, (\ref{mod-l1}) holds because:
\begin{itemize}
	\item[ ]$\hspace{0.38cm} \overline{\psi}_{M}^{2}\circ (H\otimes \overline{\psi}_{M}^{1})$
	\item [ ]$= \varphi_{M}\circ (\pi\otimes\phi_{M})\circ ((\delta_{A}\circ \pi^{-1})\otimes \varphi_{M})$ {\scriptsize (by (\ref{req-g1}))}
	\item [ ]$= \varphi_{M}\circ (\pi\otimes (\varphi_{M}\circ (\phi_{H}\otimes \phi_{M})\circ (A\otimes c_{A,H}\otimes M)\circ (\delta_{A}\otimes H\otimes M)))\circ ((\delta_{A}\circ \pi^{-1})\otimes H\otimes M) $ 
	\item[ ]$\hspace{0.38cm}${\scriptsize (by (\ref{p-v}))}
	\item [ ]$=\varphi_{M}\circ  ((\mu_{H}\circ (\pi\otimes \phi_{H})\circ (\delta_{A}\otimes \pi))\otimes \phi_{M})\circ (A\otimes c_{A,A}\otimes M)\circ ((\delta_{A}\circ \pi^{-1})\otimes \pi^{-1}\otimes M) $ {\scriptsize (by the }
	\item[ ]$\hspace{0.38cm}${\scriptsize condition of left $H$-module for $M$, the coassociativity of $\delta_{A}$, the naturality of $c$ and the condition of isomorphism}
	\item[ ]$\hspace{0.38cm}${\scriptsize  for $\pi$)}
	\item [ ]$=\varphi_{M}\circ  ((\pi\circ \mu_{A})\otimes \phi_{M})\circ (A\otimes c_{A,A}\otimes M)\circ ((\delta_{A}\circ \pi^{-1})\otimes \pi^{-1}\otimes M) $ {\scriptsize (by (\ref{1-c}))}
	\item [ ]$=\varphi_{M}\circ  (\mu_{H_{\pi}}\otimes (\phi_{M}\circ (\pi^{-1}\otimes M)))\circ (H\otimes c_{H,H}\otimes M)\circ (\delta_{H}\otimes H\otimes M) $ {\scriptsize (by the condition of coalgebra}
	\item[ ]$\hspace{0.38cm}$  {\scriptsize  isomorphism for $\pi$ and the naturality of $c$)}
	\item [ ]$=\overline{\psi}_{M}^{1}\circ (\mu_{H_{\pi}}\otimes \overline{\Gamma}_{M})\circ (H\otimes c_{H,H}\otimes M)\circ (\delta_{H}\otimes H\otimes M)$ {\scriptsize (by (\ref{pHpi1}))}
\end{itemize}

On the other hand, if $(h,l)$ is a morphisms in $\;_{(\pi, \phi_{H})}{\sf Mod}$ between \((M,N, \phi_M, \varphi_M,  \phi_N, \gamma)\) and 
\((M^{\prime},N^{\prime}, \phi_{M^{\prime}}, \varphi_{M^{\prime}},  \phi_{N^{\prime}}, \gamma^{\prime})\), we have that $h$ is a morphism in $\;_{{\mathbb H}_{\pi}}{\sf Mod}$  between $(M,\overline{\psi}_{M}^{1}, \overline{\psi}_{M}^{2})$ and $(M^{\prime},\overline{\psi}_{M^{\prime}}^{1}, \overline{\psi}_{M^{\prime}}^{2})$ because, using that $h$ is a morphism of left $H$-modules, we have $h\circ \overline{\psi}_{M}^{1}=\overline{\psi}_{M^{\prime}}^{1}\circ (H\otimes h)$ and, by (\ref{fg-g}) and the condition of morphism of left $A$-modules for $h$, we have that $h\circ \overline{\psi}_{M}^{2}=\overline{\psi}_{M^{\prime}}^{2}\circ (H\otimes h)$.

Taking into account the functors ${\sf H}_{{\sf br}}^{\pi}$, ${\sf G}_{{\mathbb H}_{\pi}}$ and ${\sf M}_{(\pi,id_{H})}$, it is easy to show that $${\sf H}_{{\sf br}}^{\pi}\circ ({\sf M}_{(\pi,id_{H})}\circ {\sf G}_{{\mathbb H}_{\pi}})	={\sf id}_{\;_{{\Bbb H}_{\pi}}{\sf Mod}}$$
and
$$(({\sf M}_{(\pi,id_{H})}\circ {\sf G}_{{\mathbb H}_{\pi}})\circ {\sf H}_{{\sf br}}^{\pi})((M,N, \phi_M, \varphi_M,  \phi_N, \gamma))=
(M,M,\phi_{M}, \varphi_{M},\overline{\phi}_{M}^{\pi}=\gamma\circ \phi_{N}\circ (A\otimes \gamma^{-1}), id_{M})$$
hold. Then, 
$$({\sf M}_{(\pi,id_{H})}\circ {\sf G}_{{\mathbb H}_{\pi}})\circ {\sf H}_{{\sf br}}^{\pi}\backsimeq {\sf id}_{_{(\pi, \phi_{H})}{\sf Mod}}$$
because $(id_{M}, \gamma)$ is an isomorphism in $_{(\pi, \phi_{H})}{\sf Mod}$ between the objects $(M,N, \phi_M, \varphi_M,  \phi_N, \gamma)$ and $(M,M,\phi_{M}, \varphi_{M},\overline{\phi}_{M}^{\pi}, id_{M}).$

\end{proof}

As a consequence of this result,  we have the following corollary whose proof is an immediate consequence of the preceding theorems.

\begin{corollary}
\label{crpHpi}
Let ${\mathbb H}$ be a Hopf brace. Then, the categories $_{(id_{H}, \Gamma_{H_{1}})}{\sf Mod}$ and $\;_{{\mathbb H}}${\sf Mod} are equivalent.
\end{corollary}

\section*{Funding}
The  authors were supported by  Ministerio de Ciencia e Innovaci\'on of Spain. Agencia Estatal de Investigaci\'on. Uni\'on Europea - Fondo Europeo de Desarrollo Regional (FEDER). Grant PID2020-115155GB-I00: Homolog\'{\i}a, homotop\'{\i}a e invariantes categ\'oricos en grupos y \'algebras no asociativas.

\bibliographystyle{amsalpha}

\end{document}